\documentclass[10pt,a4paper,twoside]{article}
\usepackage{a4} 
\usepackage{amssymb}
\usepackage{amsmath}
\usepackage{upref}  
\usepackage{microtype}
\usepackage[english]{babel}
\usepackage{hyperref}
\hypersetup{pagebackref,colorlinks,linkcolor=blue,urlcolor=black,citecolor=blue}

\newcommand{\authortitle}[2]{\author{#1}\title{#2}\markboth{#1}{#2}}

\RequirePackage{amsthm}
\newtheoremstyle{mydefinitionstyle}
	{\topsep}		
	{\topsep}		
	{\rmfamily}		
	{}				
	{\bfseries}		
	{.}				
	{ }				
	{}				
\newtheoremstyle{mytheoremstyle}
	{\topsep}		
	{\topsep}		
	{\itshape}		
	{}				
	{\bfseries}		
	{.}				
	{ }				
	{}				

\theoremstyle{mytheoremstyle}
\newtheorem{theorem}{Theorem}[section]
\newtheorem{lemma}[theorem]{Lemma}
\newtheorem{proposition}[theorem]{Proposition}
\newtheorem{corollary}[theorem]{Corollary}

\theoremstyle{mydefinitionstyle}
\newtheorem{definition}[theorem]{Definition}
\newtheorem{example}[theorem]{Example}
\newtheorem{remark}[theorem]{Remark}

\numberwithin{equation}{section}

\makeatletter
\renewenvironment{proof}[1][\proofname]{%
	\par%
	\pushQED{\qed}%
	\normalfont%
	\trivlist%
	\item[\hskip\labelsep\itshape#1\@addpunct{.}]%
	\ignorespaces%
}{%
	\popQED\endtrivlist\@endpefalse%
}
\makeatother

\makeatletter
\def\sectionmark#1{}
\def\subsectionmark#1{}
\newcommand{\sectnr}{%
	\ifnum\c@secnumdepth>\z@%
		\thesection.\hskip 1em\relax \fi%
}
\def\@evenhead{\footnotesize\rm\thepage\hfil\leftmark\hfil}
\def\@oddhead{\footnotesize\rm\hfil\rightmark\hfil\thepage}
\def\tableofcontents{\section*{Contents}\@starttoc{toc}}
\makeatother

\makeatletter
\def\@seccntformat#1{\csname the#1\endcsname.\quad}
\makeatother

\let\Enumerate=\enumerate
\let\Endenumerate=\endenumerate
\renewenvironment{enumerate}{%
	\vspace{-3pt}
	\Enumerate%
	\setlength{\itemsep}{0pt}%
	\setlength{\parskip}{0pt}%
}{%
	\Endenumerate\unskip\vspace{4pt}%
}

\makeatletter
\def\mdots@{\mathinner.\nonscript\!.%
 \ifx\next,.\else\ifx\next;.\else\ifx\next..\else
 \nonscript\!\mathinner.\fi\fi\fi}
\let\ldots\mdots@
\let\cdots\mdots@
\let\dotso\mdots@
\let\dotsb\mdots@
\let\dotsm\mdots@
\let\dotsc\mdots@
\def\vdots{\vbox{\baselineskip2.8\p@ \lineskiplimit\z@
    \kern6\p@\hbox{.}\hbox{.}\hbox{.}\kern3\p@}}
\def\ddots{\mathinner{\mkern1mu\raise8.6\p@\vbox{\kern7\p@\hbox{.}}%
    \raise5.8\p@\hbox{.}\raise3\p@\hbox{.}\mkern1mu}}
\def\cdots{\mathinner{\mkern1mu{\cdot}{\cdot}{\cdot}\mkern1mu}}
\makeatother

{\catcode`p =12 \catcode`t =12 \gdef\eeaa#1pt{#1}}
\def\accentadjtext#1{\setbox0\hbox{$#1$}\kern
                \expandafter\eeaa\the\fontdimen1\textfont1 \ht0 }
\def\accentadjscript#1{\setbox0\hbox{$#1$}\kern
                \expandafter\eeaa\the\fontdimen1\scriptfont1 \ht0 }
\def\accentadjscriptscript#1{\setbox0\hbox{$#1$}\kern
                \expandafter\eeaa\the\fontdimen1\scriptscriptfont1 \ht0 }
\def\accentadjtextback#1{\setbox0\hbox{$#1$}\kern
                -\expandafter\eeaa\the\fontdimen1\textfont1 \ht0 }
\def\accentadjscriptback#1{\setbox0\hbox{$#1$}\kern
                -\expandafter\eeaa\the\fontdimen1\scriptfont1 \ht0 }
\def\accentadjscriptscriptback#1{\setbox0\hbox{$#1$}\kern
                -\expandafter\eeaa\the\fontdimen1\scriptscriptfont1 \ht0 }
\def\itoverline#1{{\mathsurround0pt\mathchoice
        {\rlap{$\accentadjtext{\displaystyle #1}
                \accentadjtext{\vrule height1.593pt}
                \overline{\phantom{\displaystyle #1}
                \accentadjtextback{\displaystyle #1}}$}{#1}}
        {\rlap{$\accentadjtext{\textstyle #1}
                \accentadjtext{\vrule height1.593pt}
                \overline{\phantom{\textstyle #1}
                \accentadjtextback{\textstyle #1}}$}{#1}}
        {\rlap{$\accentadjscript{\scriptstyle #1}
                \accentadjscript{\vrule height1.593pt}
                \overline{\phantom{\scriptstyle #1}
                \accentadjscriptback{\scriptstyle #1}}$}{#1}}
        {\rlap{$\accentadjscriptscript{\scriptscriptstyle #1}
                \accentadjscriptscript{\vrule height1.593pt}
                \overline{\phantom{\scriptscriptstyle #1}
                \accentadjscriptscriptback{\scriptscriptstyle #1}}$}{#1}}}}
\def\itunderline#1{{\mathsurround0pt\mathchoice
        {\rlap{$\underline{\phantom{\displaystyle #1}
                \accentadjtextback{\displaystyle #1}}$}{#1}}
        {\rlap{$\underline{\phantom{\textstyle #1}
                \accentadjtextback{\textstyle #1}}$}{#1}}
        {\rlap{$\underline{\phantom{\scriptstyle #1}
                \accentadjscriptback{\scriptstyle #1}}$}{#1}}
        {\rlap{$\underline{\phantom{\scriptscriptstyle #1}
                \accentadjscriptscriptback{\scriptscriptstyle #1}}$}{#1}}}}

\def\vint{\mathop{\mathchoice%
          {\setbox0\hbox{$\displaystyle\intop$}\kern 0.22\wd0
           \vcenter{\hrule width 0.6\wd0}\kern -0.82\wd0}
          {\setbox0\hbox{$\textstyle\intop$}\kern 0.2\wd0
           \vcenter{\hrule width 0.6\wd0}\kern -0.8\wd0}
          {\setbox0\hbox{$\scriptstyle\intop$}\kern 0.2\wd0
           \vcenter{\hrule width 0.6\wd0}\kern -0.8\wd0}
          {\setbox0\hbox{$\scriptscriptstyle\intop$}\kern 0.2\wd0
           \vcenter{\hrule width 0.6\wd0}\kern -0.8\wd0}}
          \mathopen{}\int}
\newcommand{\limminus}{{\mathchoice{\raise.17ex\hbox{$\scriptstyle -$}}
                {\raise.17ex\hbox{$\scriptstyle -$}}
                {\raise.1ex\hbox{$\scriptscriptstyle -$}}
                {\scriptscriptstyle -}}}
\newcommand{\limplus}{{\mathchoice{\raise.17ex\hbox{$\scriptstyle +$}}
                {\raise.17ex\hbox{$\scriptstyle +$}}
                {\raise.1ex\hbox{$\scriptscriptstyle +$}}
                {\scriptscriptstyle +}}}

\RequirePackage{mathrsfs}
\newcommand{\calK}[0]{\mathscr{K}}
\newcommand{\calL}[0]{\mathscr{L}}
\newcommand{\calM}[0]{\mathscr{M}}
\newcommand{\calU}[0]{\mathscr{U}}
\newcommand{\calV}[0]{\mathscr{V}}

\newcommand{\R}{\mathbb{R}}

\newcommand{\eR}{{\overline{\R\kern-0.08em}\kern 0.08em}} 

\def\cprime{{\mathsurround0pt$'$}}		

\DeclareMathOperator{\Div}{div}
\DeclareMathOperator{\diam}{diam}
\DeclareMathOperator{\dist}{dist}
\DeclareMathOperator{\Lip}{Lip}
\newcommand{\Lipc}{{\Lip_c}}

\DeclareMathOperator*{\essliminf}{ess\,lim\,inf}
\DeclareMathOperator*{\essinf}{ess\,inf}

\newcommand{\setm}{\setminus}
\renewcommand{\emptyset}{\varnothing}
\newcommand{\Cp}{{C_p}}
\newcommand{\bCp}{{\itoverline{C}_p}}
\newcommand{\bCpO}{\bCp(\,\cdot\,;\Omega)}

\newcommand{\pp}{{$p\mspace{1mu}$}}   
\newcommand{\bdy}{\partial}
\newcommand{\loc}{_{\rm loc}}
\newcommand{\clOm}{\overline\Omega}
\newcommand{\clOmX}{\overline\Omega\setm\{\infty\}}
\newcommand{\bdyOmX}{\bdy\Omega\setm\{\infty\}}

\newcommand{\dmu}{d\mu}
\newcommand{\dr}{dr}
\newcommand{\dx}{dx}
\newcommand{\ds}{ds}

\newcommand{\eps}{\varepsilon}
\renewcommand{\phi}{\varphi}

\newcommand{\Lp}{L^{p}}
\newcommand{\Lploc}{L^{p}\loc}
\newcommand{\Dp}{D^p}
\newcommand{\Dploc}{D^{p}\loc}
\newcommand{\Np}{N^{1,p}}
\newcommand{\Nploc}{N^{1,p}\loc}

\newcommand{\K}{\calK}
\newcommand{\Hp}{H}
\newcommand{\U}{\calU}
\newcommand{\LL}{\calL}
\newcommand{\Pp}{P}				
\newcommand{\uPp}{\itoverline{P}}	
\newcommand{\lPp}{\itunderline{P}}	

\begin{document}

%
\authortitle{Daniel Hansevi}{%
The Perron method for \pp-harmonic functions in unbounded sets in $\R^n$ and metric spaces
}

%
\title{The Perron method for \pp-harmonic functions \\ in unbounded sets in $\R^n$ and metric spaces}
\author{%
	Daniel Hansevi \\
	\it\small Department of Mathematics, Link\"oping University, \\
	\it\small SE-581 83 Link\"oping, Sweden\/{\rm ;}
	\it\small daniel.hansevi@liu.se \\
}

\date{} 

\maketitle

\noindent{\small
 {\bf Abstract}.
The Perron method for solving 
the Dirichlet problem for \pp-harmonic functions 
is extended to unbounded open sets 
in the setting of a complete metric space 
with a doubling measure 
supporting a \pp-Poincar\'e inequality, 
$1<p<\infty$. 
The upper and lower (\pp-harmonic) Perron solutions 
are studied for open sets, 
which are assumed to be \pp-parabolic if unbounded. 
It is shown that 
continuous functions and quasicontinuous Dirichlet functions 
are resolutive 
(i.e., that their upper and lower Perron solutions coincide), 
that the Perron solution agrees with 
the \pp-harmonic extension, 
and that Perron solutions are invariant 
under perturbation of the function on a set of capacity zero.
}
\bigskip

\noindent {\small \emph{Key words and phrases}:
Dirichlet problem, 
Dirichlet space, 
doubling measure, 
metric space, 
minimal \pp-weak upper gradient, 
Newtonian space, 
nonlinear potential theory, 
obstacle problem, 
\pp-harmonic, 
\pp-parabolic, 
Perron method,
Poincar\'e inequality, 
quasicontinuity, 
upper gradient.}

\medskip


\section{Introduction}
\label{sec:intro}
The Dirichlet (boundary value) problem for \pp-harmonic functions, 
$1<p<\infty$, 
which is a nonlinear generalization of the 
classical Dirichlet problem, 
considers the \pp-Laplace equation, 
\begin{equation} \label{sec:intro-p-Laplace-eq}
	\Delta_p u 
	:= \Div(|\nabla u|^{p-2}\nabla u) 
	= 0,
\end{equation}
with prescribed boundary values $u=f$ on the boundary $\bdy\Omega$. 
A continuous weak solution of \eqref{sec:intro-p-Laplace-eq} 
is said to be \emph{\pp-harmonic}. 

The nonlinear potential theory of \pp-harmonic functions 
has been developed since the 1960s; 
not only in $\R^n$, 
but also in weighted $\R^n$, 
Riemannian manifolds, and other settings. 
The books 
Mal\'y--Ziemer~\cite{MaZi97} and 
Heinonen--Kilpel\"ainen--Martio~\cite{HeKiMa06} 
are two thorough treatments in $\R^n$ and weighted $\R^n$, respectively.

More recently, \pp-harmonic functions have been studied in 
complete metric spaces equipped with a doubling measure 
supporting a \pp-Poincar\'e inequality. 
It is not clear how to employ partial differential equations 
in such a general setting as a metric measure space. 
However, the equivalent variational problem of locally minimizing the 
\pp-energy integral, 
\begin{equation} \label{sec:intro-p-enery-eq}
	\int|\nabla u|^p\,\dx, 
\end{equation}
among all admissible functions, 
becomes available when considering 
the notion of minimal \pp-weak upper gradient as a substitute for 
the modulus of the usual gradient. 
A continuous minimizer of \eqref{sec:intro-p-enery-eq} 
is \pp-harmonic. 
The reader might want to consult 
Bj\"orn--Bj\"orn~\cite{BjBj11book} 
for the theory of \pp-harmonic functions 
and first-order analysis on metric spaces.

If the boundary value function $f$ is not continuous, 
then it is not feasible to require that the solution $u$ 
attains the boundary values as limits, 
i.e., 
to require that 
$u(y)\to f(x)$ as $y\to x$ ($y\in\Omega$) for all $x\in\bdy\Omega$. 
This is actually often not possible even if $f$ is continuous 
(see, e.g., Examples~13.3 and 13.4 in Bj\"orn--Bj\"orn~\cite{BjBj11book}). 
It is therefore more reasonable to 
consider boundary data in a weaker (Sobolev) sense. 
Shanmugalingam~\cite{Shanmugalingam01} 
solved the Dirichlet problem for \pp-harmonic functions 
in bounded domains 
with Newtonian boundary data 
taken in Sobolev sense. 
This result was generalized by Hansevi~\cite{Hansevi15} to 
unbounded domains with Dirichlet boundary data.
For continuous boundary values, 
the problem was solved in bounded domains using 
uniform approximation by 
Bj\"orn--Bj\"orn--Shanmugalingam~\cite{BjBjSh03a}.

The Perron method for solving the Dirichlet problem 
for harmonic functions (on $\R^2$) was introduced in 1923 
by Perron~\cite{Perron23} 
(and independently by Remak~\cite{Remak24}). 
The advantage of the method is that one can construct reasonable 
solutions for arbitrary boundary data. 
It provides an upper and a lower solution, 
and the major question is to determine when these solutions coincide, 
i.e., to determine when the boundary data is \emph{resolutive}. 
The Perron method in connection with the usual Laplace operator 
has been studied extensively in Euclidean domains 
(see, e.g., Brelot~\cite{Brelot39} for the 
complete characterization of the resolutive functions) 
and has been extended to degenerate elliptic operators 
(see, e.g., 
Granlund--Lindqvist--Martio~\cite{GrLiMa86}, 
Kilpel\"ainen~\cite{Kilpelainen89}, 
and Heinonen--Kilpel\"ainen--Martio~\cite{HeKiMa06}). 

Bj\"orn--Bj\"orn--Shanmugalingam~\cite{BjBjSh03b} 
extended the Perron method for \pp-harmonic functions to the setting of 
a complete metric space equipped with a doubling measure 
supporting a \pp-Poincar\'e inequality, 
and proved that 
Perron solutions are \pp-harmonic and agree with 
the previously obtained solutions for Newtonian boundary data 
in Shanmugalingam~\cite{Shanmugalingam01}. 
More recently, 
Bj\"orn--Bj\"orn--Shanmugalingam~\cite{BjBjSh13a} have 
developed the Perron method for \pp-harmonic functions 
with respect to the Mazurkiewicz boundary. 
See also 
Estep--Shanmugalingam~\cite{EsSh15}, 
A.~Bj\"orn~\cite{BjornA15}, 
and Bj\"orn--Bj\"orn--Sj\"odin~\cite{BjBjSj16}.

The purpose of this paper is to extend the Perron method 
for solving the Dirichlet problem for \pp-harmonic functions to 
\emph{unbounded} open sets 
in the setting of 
a complete metric space equipped with a doubling measure 
supporting a \pp-Poincar\'e inequality. 
In particular, we show that 
quasicontinuous functions with finite Dirichlet energy, 
as well as continuous functions, 
are resolutive with respect to open sets, 
which are assumed to be \pp-parabolic if unbounded,   
and that the 
Perron solution is the unique \pp-harmonic solution 
that takes the required boundary data 
outside sets of capacity zero. 
We also show that Perron solutions 
are invariant under perturbations on sets of capacity zero. 

The paper is organized as follows: 
In the next section, 
we establish notation, 
review some basic definitions relating to Sobolev-type spaces on metric spaces, 
and obtain a new convergence lemma. 
In Section~\ref{sec:obstacle-problem}, 
we review the obstacle problem associated with \pp-harmonic functions 
in unbounded sets 
and obtain a convergence theorem 
that will be important in the proof of Theorem~\ref{thm:Dp-resolutive} 
(the main result of this paper). 
Section~\ref{sec:p-parabolic} is devoted to \pp-parabolic sets. 
The necessary background on \pp-harmonic and 
superharmonic functions is given in Section~\ref{sec:p-harmonic}, 
making it possible to define Perron solutions in Section~\ref{sec:Perron}, 
where we also extend the comparison principle 
for superharmonic functions to unbounded sets. 
In Section~\ref{sec:resolutivity}, 
we introduce a smaller capacity 
(and its related quasicontinuity property) 
before we obtain our main result 
(Theorem~\ref{thm:Dp-resolutive}) on resolutivity 
(of quasicontinuous functions) 
along with some consequences.

\section{Notation and preliminaries}
\label{sec:prel} 
We assume throughout the paper that $(X,\calM,\mu,d)$ 
is a metric measure space (which we refer to as $X$) 
equipped with a metric $d$ and a 
positive complete Borel measure $\mu$ such that 
$0<\mu(B)<\infty$ 
for all balls $B\subset X$. 
We use the following notation for balls, 
\[
	B(x_0,r) 
	:= \{x\in X:d(x,x_0)<r\},
\]
and for $B=B(x_0,r)$ and $\lambda>0$, 
we let $\lambda B=B(x_0,\lambda r)$.
The $\sigma$-algebra $\calM$ (on which $\mu$ is defined) 
is the completion of the Borel $\sigma$-algebra. 
Later we will impose additional requirements 
on the space and on the measure. 
We assume further that 
$1<p<\infty$ 
and that\/ $\Omega$ is a nonempty 
\textup{(}possibly unbounded\textup{)} 
open subset of $X$. 

The measure $\mu$ is said to be \emph{doubling} if there exists 
a constant $C\geq 1$ such that 
\[
	0 < \mu(2B) \leq C \mu(B) < \infty 
\]
for all balls $B\subset X$. 
Recall that a metric space is said to be \emph{proper} 
if all bounded closed subsets are compact. 
In particular, 
this is true if the metric space is complete and the measure is doubling. 

The characteristic function 
of a set $E$ is denoted by $\chi_E$, 
and we let $\sup\emptyset=-\infty$ 
and $\inf\emptyset=\infty$. 
We say that the set $E$ is compactly contained in $A$ if $\itoverline{E}$ 
(the closure of $E$) is a compact subset of $A$ 
and denote this by $E\Subset A$. 
The extended real number system is denoted by 
$\eR:=[-\infty, \infty]$. 
We use the notation $f_\limplus=\max\{f, 0\}$ and
$f_\limminus=\max\{-f,0\}$.
Continuous functions will be assumed to be real-valued. 
By a curve in $X$ 
we mean a rectifiable nonconstant continuous mapping 
from a compact interval into $X$. 
A curve can thus be parametrized by its arc length $\ds$. 
\begin{definition}\label{def:upper-gradients}
A Borel function $g\colon X\to[0,\infty]$ is said to be an 
\emph{upper gradient} of a function $f\colon X\to\eR$ whenever 
\begin{equation}\label{def:upper-gradients-ineq}
	|f(x)-f(y)|
	\leq \int_\gamma g\,\ds
\end{equation}
holds for each pair of points 
$x,y\in X$ and every curve $\gamma$ in $X$ joining $x$ and $y$. 
We make the convention that the left-hand side is infinite 
when at least one of the terms in the left-hand side is infinite.
\end{definition}
A drawback of the upper gradients, 
introduced in 
Heinonen--Koskela~\cite{HeKo96},\cite{HeKo98}, 
is that they are not preserved by $\Lp$-convergence. 
It is, however, possible to overcome this problem 
by relaxing the condition a bit (Koskela--MacManus~\cite{KoMac98}). 
\begin{definition}\label{def:p-weak-upper-gradients}
A measurable function $g\colon X\to[0,\infty]$ is said to be a 
\emph{\pp-weak upper gradient} of a function $f\colon X\to\eR$ whenever 
\eqref{def:upper-gradients-ineq} holds for each pair of points 
$x,y\in X$ and \pp-almost every curve (see below) $\gamma$ in $X$ 
joining $x$ and $y$. 
\end{definition}
Note that a \pp-weak upper gradient is not required to be a Borel function 
(see the discussion in the notes to Chapter 1 in Bj\"orn--Bj\"orn~\cite{BjBj11book}).

We say that a property holds for \emph{\pp-almost every curve} 
if it fails only for 
a curve family $\Gamma$ with zero \pp-modulus, 
i.e., 
if there exists a nonnegative $\rho\in\Lp(X)$ such that 
$\int_\gamma\rho\,\ds=\infty$ for every curve $\gamma\in\Gamma$. 

A countable union of curve families, each with zero \pp-modulus, 
also has zero \pp-modulus. 
For proofs of this and other results in this section, 
we refer to Bj\"orn--Bj\"orn~\cite{BjBj11book} or 
Heinonen--Koskela--Shanmugalingam--Tyson~\cite{HeKoShTy15}. 

Shanmugalingam~\cite{Shanmugalingam00} 
used upper gradients to 
define so-called Newtonian spaces. 
\begin{definition}\label{def:Newtonian-space}
The \emph{Newtonian space} on $X$, 
denoted by $\Np(X)$, 
is the space of all 
everywhere defined, extended real-valued functions $u\in\Lp(X)$ 
such that  
\[
	\|u\|_{\Np(X)} 
	:= \biggl(\int_X|u|^p\,\dmu + \inf_g\int_X g^p\,\dmu\biggr)^{1/p}<\infty, 
\]
where the infimum is taken over all upper gradients $g$ of $u$.
\end{definition}
\begin{definition}\label{def:Dirichlet-space}
An everywhere defined, 
measurable, 
extended real-valued function on $X$ 
belongs to the 
\emph{Dirichlet space} $\Dp(X)$ 
if it has an upper gradient in $\Lp(X)$.
\end{definition}
It follows from Lemma~2.4 in Koskela--MacManus~\cite{KoMac98} 
that a measurable function belongs to $\Dp(X)$ whenever 
it (merely) has a \pp-weak upper gradient in $\Lp(X)$.

We emphasize that 
Newtonian and Dirichlet functions are defined \emph{everywhere} 
(not just up to an equivalence class in the
corresponding function space), 
which is essential for the notion of upper gradient to make sense. 
Shanmugalingam~\cite{Shanmugalingam00} proved that 
the associated normed (quotient) space defined by 
$\Np(X)/\sim$, where $u\sim v$ 
if and only if $\|u-v\|_{\Np(X)}=0$, 
is a Banach space. 

A measurable set $A\subset X$ can be 
considered to be a metric space in its own right
(with the restriction of $d$ and $\mu$ to $A$). 
Thus the Newtonian space $\Np(A)$ and the Dirichlet space $\Dp(A)$ 
are also given by 
Definitions~\ref{def:Newtonian-space}~and~\ref{def:Dirichlet-space}, 
respectively. 
If $X$ is proper, 
then $f\in\Lploc(\Omega)$, $f\in\Nploc(\Omega)$, and $f\in\Dploc(\Omega)$ 
if and only if 
$f\in\Lp(\Omega')$, $f\in\Np(\Omega')$, and $f\in\Dp(\Omega')$, 
respectively, for all open $\Omega'\Subset\Omega$. 

If $u\in\Dp(X)$, 
then $u$ has a \emph{minimal \pp-weak upper gradient}, 
denoted by $g_u$, which is minimal in the sense that $g_u\leq g$ a.e.\ for 
all \pp-weak upper gradients $g$ of $u$; 
see Shanmugalingam~\cite{Shanmugalingam01}. 
Minimal \pp-weak upper gradients $g_u$ are true substitutes 
for $|\nabla u|$ in metric spaces.
One of the important properties of minimal \pp-weak upper gradients 
is that they are
local in the sense that if two functions 
$u,v\in\Dp(X)$ coincide on a set $E$, 
then $g_u=g_v$ a.e.\ on $E$. 
Furthermore, if $U=\{x\in X:u(x)>v(x)\}$, 
then $g_u\chi_U+g_v\chi_{X\setm U}$ and 
$g_v\chi_U+g_u\chi_{X\setm U}$ 
are minimal \pp-weak upper gradients of $\max\{u,v\}$ 
and $\min\{u,v\}$, respectively. 
The restriction of a minimal \pp-weak upper gradient 
to an open subset remains minimal with respect to that subset, 
and hence the results above about minimal \pp-weak upper gradients 
of functions in $\Dp(X)$  
extend to functions in $\Dploc(X)$ 
having minimal \pp-weak upper gradients in $\Lploc(X)$. 

The notion of capacity of a set is important in potential theory, 
and various types and definitions can be found in the literature 
(see, e.g., Kinnunen--Martio~\cite{KiMa96} and 
Shanmugalingam~\cite{Shanmugalingam00}). 
\begin{definition}\label{def:capacity}
Let $A\subset X$ be measurable. 
The (\emph{Sobolev}) \emph{capacity} (with respect to $A$) 
of $E\subset A$ is the number 
\[
	\Cp(E;A) 
	:= \inf_u\|u\|_{\Np(A)}^p,
\]
where the infimum is taken over all 
$u\in\Np(A)$ such that $u\geq 1$ on $E$. 
When the capacity is taken with respect to $X$, 
we simplify the notation and write $\Cp(E)$.

Whenever a property holds for all points 
except for those in a set of capacity zero, 
it is said to hold \emph{quasieverywhere} (\emph{q.e.}). 
\end{definition}
The capacity is countably subadditive, 
i.e., 
$\Cp(\bigcup_{j=1}^\infty E_j)\leq\sum_{j=1}^\infty\Cp(E_j)$. 

In order to be able to compare boundary values 
of Dirichlet and Newtonian functions, 
we introduce the following spaces.
\begin{definition}\label{def:Dp0}
For subsets $E$ and $A$ of $X$, 
where $A$ is measurable, 
the \emph{Dirichlet space with zero boundary values in $A\setm E$}, 
is 
\[
	\Dp_0(E;A) 
	:= \{u|_{E\cap A}:u\in\Dp(A)\textup{ and }u=0\textup{ in }A\setm E\}.
\]
The \emph{Newtonian space with zero boundary values}, $\Np_0(E;A)$, is defined analogously. 
We let $\Dp_0(E)$ and $\Np_0(E)$ denote 
$\Dp_0(E;X)$ and $\Np_0(E;X)$, respectively.
\end{definition}
The condition ``$u=0$ in $A\setm E$'' 
can actually be replaced by ``$u=0$ q.e.\ in $A\setm E$'' 
without changing the obtained spaces. 

If $E\subset X$ is measurable, $f\in\Dp(E)$, 
$f_1,f_2\in\Dp_0(E)$, and $f_1\leq f\leq f_2$ q.e.\ in $E$\textup{,} 
then $f\in\Dp_0(E)$ (this is Lemma~2.8 in Hansevi~\cite{Hansevi15}). 

The following convergence lemma 
will be used to prove Theorem~\ref{thm:obst-lim}, 
which in turn will be important when we prove 
Theorem~\ref{thm:Dp-resolutive}. 
\begin{lemma}\label{lem:conv-grad}
Let $G_1,G_2,\dots$ be open sets such that\/ 
$G_1\subset G_2\subset\cdots\subset X=\bigcup_{k=1}^\infty G_k$ 
and let\/ $\{u_j\}_{j=1}^\infty$ be a sequence of functions defined on $X$.  
Assume that\/ $\{u_j\}_{j=1}^\infty$ 
is bounded in $\Lp(G_k)$ 
for all $k=1,2,\dots$\,. 
Assume further that\/ 
$\{g_j\}_{j=1}^\infty$ 
is bounded in $\Lp(X)$\textup{,} 
and that $g_j$ is a \pp-weak upper gradient of $u_j$ 
with respect to $G_j$ for each $j=1,2,\dots$\,. 
Then a function $u$ belongs to $\Dp(X)$ 
if $u_j\to u$ q.e.\ on $X$ as $j\to\infty$.
\end{lemma}
\begin{proof}
Let $k$ be a positive integer. 
Clearly, 
$g_j$ is a \pp-weak upper gradient of $u_j$ 
with respect to $G_k$ for every integer $j\geq k$. 
According to Lemma~3.2 in 
Bj\"orn--Bj\"orn--Parviainen~\cite{BjBjPa10}, 
there are a \pp-weak upper gradient 
$\tilde g_k\in\Lp(G_k)$ of $u$ with respect to $G_k$  
and a subsequence of $\{g_j\}_{j=1}^\infty$, 
denoted by $\{g_{k,j}\}_{j=1}^\infty$,  
such that $g_{k,j}\to\tilde g_k$ weakly in $\Lp(G_k)$ as $j\to\infty$. 
Extend $\tilde g_k$ to $X$ by letting $\tilde g_k=0$ on $X\setm G_k$. 
Since $\{g_j\}_{j=1}^\infty$ is bounded in $\Lp(X)$, 
there is an integer $M$ such that $\|g_j\|_{\Lp(X)}\leq M$ 
for all $j=1,2,\dots$\,. 
The weak convergence implies that 
\[
	\|\tilde g_k\|_{\Lp(X)} 
	= \|\tilde g_k\|_{\Lp(G_k)} 
	\leq \liminf_{j\to\infty}\|g_{k,j}\|_{\Lp(G_k)} 
	\leq \liminf_{j\to\infty}\|g_{k,j}\|_{\Lp(X)} 
	\leq M,  
\]
and hence the sequence 
$\{\tilde g_k\}_{k=1}^\infty$ is bounded in $\Lp(X)$. 

Since $\Lp(X)$ is reflexive, 
it follows from Banach--Alaoglu's theorem 
that there is a subsequence, 
also denoted by $\{\tilde g_k\}_{k=1}^\infty$, 
that converges weakly in $\Lp(X)$ 
to a function $g$. 
By applying Mazur's lemma (see, e.g., Theorem~3.12 in Rudin~\cite{Rudin91}) 
repeatedly to the 
sequences $\{\tilde g_k\}_{k=j}^\infty$, $j=1,2,\dots$\,, 
we can find convex combinations 
\[
	g'_j 
	= \sum_{k=j}^{N_j} a_{j,k}\tilde g_k 
\]
such that $\|g'_j-g\|_{\Lp(X)}<1/j$, 
and hence we obtain a sequence 
$\{g'_j\}_{j=1}^\infty$ 
that converges to $g$ in $\Lp(X)$. 
Note that 
$g\in\Lp(X)$, 
and that for every $n=1,2,\dots$\,, 
the sequence $\{g'_j\}_{j=n}^\infty$ 
consists of \pp-weak upper gradients of $u$ 
with respect to $G_n$. 
It suffices to show that $g$ is a \pp-weak upper gradient of $u$ 
to complete the proof. 

By Fuglede's lemma (Lemma~3.4 in Shanmugalingam~\cite{Shanmugalingam00}), 
we can find a subsequence,  
also denoted by $\{g'_j\}_{j=1}^\infty$,  
and a collection of curves $\Gamma$ in $X$ 
with zero \pp-modulus, 
such that for every curve $\gamma\notin\Gamma$, 
it follows that  
\begin{equation}\label{lem:conv-grad-fuglede}
	\int_\gamma g'_j\,\ds 
	\to \int_\gamma g\,\ds
	\quad\textup{as }j\to\infty.
\end{equation}

For every $n=1,2,\dots$\,, 
let $\Gamma_{n,j}$, $j=n,n+1,\dots$\,, be the collection of curves in $G_n$ 
along which $g'_j$ is not an upper gradient of $u$, 
and let 
\[
	\Gamma' 
	= \Gamma\cup\bigcup_{n=1}^\infty\bigcup_{j=n}^\infty\Gamma_{n,j}.
\]
Then $\Gamma'$ has zero \pp-modulus. 

Let $\gamma\notin\Gamma'$ 
be an arbitrary curve in $X$ with endpoints $x$ and $y$. 
Since $\gamma$ is compact and $G_1,G_2,\dots$\, 
are open sets that exhaust $X$, 
we can find an integer $N$ such that $\gamma\subset G_N$ 
and 
\[
	|u(x)-u(y)| 
	\leq \int_\gamma g'_j\,\ds,
	\quad j=N,N+1,\dots.
\]
It follows that $g$ is a \pp-weak upper gradient of $u$, 
and thus $u\in\Dp(X)$, since 
\[
	|u(x)-u(y)| 
	\leq \lim_{j\to\infty}\int_\gamma g'_j\,\ds 
	= \int_\gamma g\,\ds. 
	\qedhere
\]
\end{proof}
\begin{definition}\label{def:Poincare-inequality}
Let $q\geq 1$. 
We say that $X$ supports a $(q,p)$-\emph{Poincar\'e inequality} 
if there exist constants, 
$C>0$ and $\lambda\geq 1$ (the dilation constant), 
such that 
\begin{equation}\label{def:Poincare-inequality-ineq}
	\biggl(\vint_B|u-u_B|^q\,\dmu\biggr)^{1/q} 
	\leq C\diam(B)
		\biggl(\vint_{\lambda B}g^p\,\dmu\biggr)^{1/p}
\end{equation}
for all balls $B\subset X$, 
all integrable functions $u$ on $X$, 
and all upper gradients $g$ of $u$. 
\end{definition}
In \eqref{def:Poincare-inequality-ineq}, 
we have used the convenient notation 
$u_B := \vint_B u\,\dmu := \frac{1}{\mu(B)}\int_B u\,\dmu$. 
We usually write \pp-\emph{Poincar\'e inequality} 
instead of $(1,p)$-Poincar\'e inequality. 

Requiring a Poincar\'e inequality to hold is one way of 
making it possible to control functions by their upper gradients. 

\section{The obstacle problem} 
\label{sec:obstacle-problem} 
\emph{In this section\textup{,} 
we also assume that $X$ is proper and 
supports a $(p,p)$-Poincar\'e inequality\textup{,}  
and that $\Cp(X\setm\Omega)>0$.}

\bigskip
Inspired by Kinnunen--Martio~\cite{KiMa02}, 
the following obstacle problem, 
which is a generalization that allows for unbounded sets, 
was defined in Hansevi~\cite{Hansevi15}.
\begin{definition}\label{def:obst}
Let $V\subset X$ be a nonempty open subset with $\Cp(X\setm V)>0$. 
For $\psi\colon V\to\eR$ and $f\in\Dp(V)$, 
define 
\[
	\K_{\psi,f}(V)
	= \{v\in\Dp(V):v-f\in\Dp_0(V)
		\textup{ and }v\geq\psi\text{ q.e.\ in }V\}.
\]
A function $u$ is said to be a 
\emph{solution of the }$\K_{\psi,f}(V)$-\emph{obstacle problem 
\textup{(}with obstacle $\psi$ and boundary values $f$\,\textup{)}}
whenever $u\in\K_{\psi,f}(V)$ and 
\[
	\int_V g_u^p\,\dmu 
	\leq \int_V g_v^p\,\dmu
	\quad\textup{for all }v\in\K_{\psi,f}(V).
\]
When $V=\Omega$, 
we usually denote $\K_{\psi,f}(\Omega)$ by $\K_{\psi,f}$ for short.
\end{definition}
It was proved in Hansevi~\cite{Hansevi15} that 
the $\K_{\psi,f}$-obstacle problem has a unique 
(up to sets of capacity zero) solution
under the natural condition of 
$\K_{\psi,f}$ being nonempty. 
If the measure $\mu$ is doubling, 
then there is a unique lsc-regularized solution of the 
$\K_{\psi,f}$-obstacle problem whenever $\K_{\psi,f}$ is nonempty 
(Theorem~4.1 in Hansevi~\cite{Hansevi15}). 
The \emph{lsc-regularization} of $u$ is the 
(lower semicontinuous) function $u^*$ defined by 
\[
	u^*(x) 
	= \essliminf_{y\to x}u(y) 
	:= \lim_{r\to 0}\essinf_{B(x,r)} u.
\]

We conclude this section with a proof of a 
new convergence theorem that will be used in 
the proof of Theorem~\ref{thm:Dp-resolutive}. 
It is a generalization of 
Proposition~10.18 in Bj\"orn--Bj\"orn~\cite{BjBj11book} 
to unbounded sets and Dirichlet functions. 
The special case when $\psi_j=f_j\in\Np(\Omega)$ 
had previously been proved in 
Kinnunen--Shanmugalingam~\cite{KiSh06}, 
and a similar result for the double obstacle problem 
was obtained in Farnana~\cite{Farnana10a}. 
\begin{theorem}\label{thm:obst-lim}
Let\/ $\{\psi_j\}_{j=1}^\infty$ and\/ $\{f_j\}_{j=1}^\infty$ 
be sequences of functions in $\Dp(\Omega)$ 
that are decreasing q.e.\ to functions 
$\psi$ and $f$ in $\Dp(\Omega)$\textup{,} respectively\textup{,} 
and are such that\/ 
$\|g_{\psi_j-\psi}\|_{\Lp(\Omega)}\to 0$ and\/ 
$\|g_{f_j-f}\|_{\Lp(\Omega)}\to 0$ as $j\to\infty$. 
If $u_j$ is a solution of the $\K_{\psi_j,f_j}$-obstacle problem 
for each $j=1,2,\dots$\,,
then the sequence\/ $\{u_j\}_{j=1}^\infty$ is decreasing q.e.\ in $\Omega$ 
to a function which is a solution of the $\K_{\psi,f}$-obstacle problem.
\end{theorem}
\begin{proof}
The comparison principle (Lemma~3.6 in Hansevi~\cite{Hansevi15}) asserts that 
$u_{j+1}\leq u_j$ q.e.\ in $\Omega$ for each $j=1,2,\dots$\,, 
and hence by the subadditivity of the capacity there exists a function $u$ 
such that $\{u_j\}_{j=1}^\infty$ is decreasing 
to $u$ q.e.\ in $\Omega$. 
We will show that $u$ is a solution of the $\K_{\psi,f}$-obstacle problem. 

Let $w_j=u_j-f_j$ and $w=u-f$, 
all functions extended by zero outside $\Omega$. 
Let $B\subset X$ be a ball such that $B\cap\Omega$ is nonempty 
and $\Cp(B'\setm\Omega)>0$ where  
$B':=\tfrac{1}{2}B$. 

We claim that the sequences 
$\{g_{w_j}\}_{j=1}^\infty$ and $\{w_j\}_{j=1}^\infty$  
are bounded in 
$\Lp(X)$ and $\Lp(kB)$, respectively, 
for every $k=1,2,\dots$\,. 
To show this, let $k$ be a positive integer. 
Let \smash{$S=\bigcap_{j=1}^\infty S_j$}, 
where $S_j:=\{x\in X:w_j(x)=0\}$. 
Proposition~4.14 in Bj\"orn--Bj\"orn~\cite{BjBj11book} 
asserts that \smash{$w_j\in\Nploc(X)$}, 
and since 
\[
	\Cp(kB'\cap S_j)
	\geq \Cp(kB'\cap S)
	\geq \Cp(kB'\setm\Omega)
	\geq \Cp(B'\setm\Omega)
	> 0, 
\]
Maz\cprime ya's inequality (Theorem~5.53 in Bj\"orn--Bj\"orn~\cite{BjBj11book}) 
implies the existence of constants $C_{kB,\Omega}>0$ 
and $\lambda\geq 1$ such that
\[
	\int_{kB}|w_j|^p\,\dmu 
	\leq C_{kB,\Omega}\int_{\lambda kB}g_{w_j}^p\,\dmu. 
\]
Let $h_j=\max\{f_j,\psi_j\}$. 
Then 
$0\leq h_j-f_j=(\psi_j-f_j)_\limplus\leq(u_j-f_j)_\limplus$ 
q.e.\ in $\Omega$, 
and hence 
Lemma~2.8 in Hansevi~\cite{Hansevi15} asserts that 
$h_j-f_j\in\Dp_0(\Omega)$. 
Clearly, $h_j\in\K_{\psi_j,f_j}$, and 
since $u_j$ is a solution of the $\K_{\psi_j,f_j}$-obstacle problem, 
it follows that $\|g_{u_j}\|_{\Lp(\Omega)}\leq\|g_{h_j}\|_{\Lp(\Omega)}$. 
We also know that $g_{h_j}\leq g_{\psi_j}+g_{f_j}$ a.e.\ in $\Omega$, 
and therefore the claim follows because 
\begin{align}
	C_{kB,\Omega}^{\,-1/p}\|w_j\|_{\Lp(kB)} 
	&\leq \|g_{w_j}\|_{\Lp(X)} \nonumber \\
	&\leq \|g_{u_j}\|_{\Lp(\Omega)} 
		+ \|g_{f_j}\|_{\Lp(\Omega)} \nonumber \\
	&\leq \|g_{h_j}\|_{\Lp(\Omega)} 
		+ \|g_{f_j}\|_{\Lp(\Omega)} \label{thm:obst-lim-ineq} \\ 
	&\leq \|g_{\psi_j}\|_{\Lp(\Omega)} 
		+ 2\|g_{f_j}\|_{\Lp(\Omega)}  \nonumber \\
	&\leq \|g_{\psi_j-\psi}\|_{\Lp(\Omega)} 
		+ \|g_\psi\|_{\Lp(\Omega)}  
		+ 2\|g_{f_j-f}\|_{\Lp(\Omega)} 
		+ 2\|g_{f}\|_{\Lp(\Omega)} \nonumber.
\end{align}
Lemma~\ref{lem:conv-grad} applies here and asserts that $w\in\Dp(X)$, 
and hence $u-f\in\Dp_0(\Omega)$. 
Because $f\in\Dp(\Omega)$, 
this also shows that $u\in\Dp(\Omega)$. 
Since $\Cp$ is countably subadditive, 
$u\geq\psi$ q.e.\ in $\Omega$, 
and hence $u\in\K_{\psi,f}$. 

Let $v$ be an arbitrary function that belongs to $\K_{\psi,f}$. 
We complete the proof by showing that 
\begin{equation}\label{thm:obst-lim-ineq-minimizer}
	\int_\Omega g_u^p\,\dmu
	\leq \int_\Omega g_v^p\,\dmu.
\end{equation}
Let $\phi_j=\max\{v+f_j-f,\psi_j\}$. 
Clearly, 
$\phi_j\geq\psi_j$ 
and $\phi_j\in\Dp(\Omega)$. 
Furthermore, 
\[
	v-f 
	\leq \max\{v-f,\psi_j-f_j\}
	= \phi_j-f_j
	\leq \max\{v-f,(u_j-f_j)_\limplus\}
	\quad\text{q.e.\ in }\Omega,
\] 
and hence $\phi_j-f_j\in\Dp_0(\Omega)$ by 
Lemma~2.8 in Hansevi~\cite{Hansevi15}. 
We conclude that $\phi_j\in\K_{\psi_j,f_j}$,
and therefore 
\[
	\int_\Omega g_{u_j}^p\,\dmu 
	\leq \int_\Omega g_{\phi_j}^p\,\dmu.
\]

Let $E$ be the set where 
$\{f_j\}_{j=1}^\infty$ decreases to $f$, 
$\{\psi_j\}_{j=1}^\infty$ decreases to $\psi$, 
and simultaneously $v\geq\psi$. 
Then $\Cp(\Omega\setm E)=0$.

Let $U_j=\{x\in E:(f_j-f)(x)<(\psi_j-v)(x)\}$. 
Clearly, $\phi_j-v=\psi_j-v$ in $U_j$ 
and $\phi_j-v=f_j-f$ in $E\setm U_j$, 
and hence it follows that  
\begin{align}
	\int_\Omega g_{\phi_j-v}^p\,d\mu
	&\leq \int_{U_j}(g_{\psi_j-\psi}+g_{\psi-v})^p\,d\mu
	 + \int_{E\setm U_j}g_{f_j-f}^p\,d\mu \nonumber \\
	&\leq 2^p\int_{U_j}g_{\psi-v}^p\,d\mu
	  + 2^p\int_\Omega g_{\psi_j-\psi}^p\,d\mu
	  + \int_\Omega g_{f_j-f}^p\,d\mu, \label{thm:obst-lim-ineq-2}
\end{align}
where the last two integrals tend to zero as $j\to\infty$.

Let $V_j=\{x\in E:\psi(x)<v(x)<\psi_j(x)\}$. 
Since $f_j-f\geq 0$ in $E$, 
we know that $v<\psi_j$ in $U_j$, 
and because $g_{\psi-v}=0$ a.e.\ in 
\[
	\{x\in E:v(x)\leq\psi(x)\}=\{x\in E:v(x)=\psi(x)\},
\] 
it follows that   
\begin{equation}
	\int_{U_j}g_{\psi-v}^p\,d\mu 
	\leq \int_{V_j}g_{\psi-v}^p\,d\mu. \label{thm:obst-lim-ineq-3}
\end{equation}
The fact that $\{\psi_j\}_{j=1}^\infty$ is decreasing to $\psi$ in $E$ 
implies that 
$g_{\psi-v}\chi_{V_j}\to 0$ everywhere in $E$ as $j\to\infty$, 
and since 
$|g_{\psi-v}\chi_{V_j}|\leq g_{\psi-v}\leq g_\psi+g_v$ a.e.\ in $E$ 
and $g_\psi+g_v\in\Lp(E)$, 
dominated convergence asserts that 
\begin{equation}
	\int_{V_j}g_{\psi-v}^p\,d\mu 
	= \int_E g_{\psi-v}^p\chi_{V_j}\,d\mu\to 0
	\quad\textup{as }j\to\infty. \label{thm:obst-lim-eq-vanishing}
\end{equation}
It follows from 
\eqref{thm:obst-lim-ineq-2}, 
\eqref{thm:obst-lim-ineq-3}, and 
\eqref{thm:obst-lim-eq-vanishing} 
that 
$g_{\phi_j}\to g_v$ in $\Lp(\Omega)$ as $j\to\infty$.

Let 
\[
	\Omega_k 
	= \{x\in kB\cap\Omega:\dist(x,\bdy\Omega)>\delta/k\}, 
	\quad k=1,2,\dots,
\] 
where $\delta>0$ is sufficiently small so that $\Omega_1$ is nonempty. 
It is clear that 
\[
	\Omega_1\Subset\Omega_2\Subset\cdots\Subset\Omega
	= \bigcup_{k=1}^\infty\Omega_k.
\]
Fix a positive integer $k$. 
Then $g_u$ and $g_{u_j}$ are minimal \pp-weak upper gradients of 
$u$ and $u_j$, respectively, with respect to $\Omega_k$. 
By Proposition~4.14 in Bj\"orn--Bj\"orn~\cite{BjBj11book}, 
the functions $f$ and $f_j$ belong to $\Lploc(\Omega)$, 
and hence $f$ and $f_j$ are in $\Lp(\Omega_k)$. 
Furthermore, 
$\{f_j\}_{j=1}^\infty$ is decreasing to $f$ q.e.\ in $\Omega$, 
and therefore $|f_j-f|\leq|f_1-f|$ q.e.\ in $\Omega$. 
By \eqref{thm:obst-lim-ineq}, 
we can see that 
$\{w_j\}_{j=1}^\infty$ is bounded in $\Lp(kB)$, 
and also that 
$\{g_{u_j}\}_{j=1}^\infty$ is bounded in $\Lp(\Omega)$. 
Since 
\[
	\|u_j\|_{\Lp(\Omega_k)} 
	\leq \|w_j\|_{\Lp(kB)} 
		+ \|f_1-f\|_{\Lp(\Omega_k)} + \|f\|_{\Lp(\Omega_k)}, 
\]
it follows that 
$\{u_j\}_{j=1}^\infty$ is bounded in $\Np(\Omega_k)$, 
and because $u_j\to u$ q.e.\ in $\Omega$ as $j\to\infty$, 
Corollary~3.3 in Bj\"orn--Bj\"orn--Parviainen~\cite{BjBjPa10} asserts that 
\[
	\int_{\Omega_k}g_u^p\,d\mu 
	\leq \liminf_{j\to\infty}\int_{\Omega_k}g_{u_j}^p\,d\mu 
	\leq \liminf_{j\to\infty}\int_\Omega g_{u_j}^p\,d\mu 
	\leq \liminf_{j\to\infty}\int_\Omega g_{\phi_j}^p\,d\mu
	= \int_\Omega g_v^p\,d\mu.
\]
Letting $k\to\infty$ 
yields \eqref{thm:obst-lim-ineq-minimizer}
and the proof is complete.
\end{proof}

If $\mu$ is doubling, 
then $X$ is proper if and only if $X$ is complete 
(see, e.g., Proposition~3.1 in Bj\"orn--Bj\"orn~\cite{BjBj11book}). 
H\"older's inequality implies that 
$X$ supports a \pp-Poincar\'e inequality 
if $X$ supports a $(p,p)$-Poincar\'e inequality. 
The converse is true when $\mu$ is doubling;  
see Theorem~5.1 in Haj\l{}asz--Koskela~\cite{HaKo00}. 
Thus adding the assumption that $\mu$ is doubling 
leads to the rather standard assumptions stated below.

\medskip
\emph{We assume from now on 
that\/ $1<p<\infty$\textup{,} 
that $X$ is a complete metric measure space 
supporting a \pp-Poincar\'e inequality\textup{,}  
that $\mu$ is doubling\textup{,} 
and that\/ $\Omega\subset X$ 
is a nonempty \textup{(}possibly unbounded\textup{)} 
open subset with $\Cp(X\setm\Omega)>0$.}

\section{\texorpdfstring{\boldmath$p\mspace{1mu}$}{p}-parabolicity}
\label{sec:p-parabolic} 
\emph{Note the standing assumptions described at the end of the previous section.}

\bigskip
In the proof of Theorem~\ref{thm:Dp-resolutive}, 
we need $\Omega$ to be \pp-parabolic if it is unbounded.
\begin{definition}\label{def:p-parabolicity}
If $\Omega$ is unbounded, 
then we say that $\Omega$ is \emph{\pp-parabolic} if 
for every compact $K\subset\Omega$, 
there exist functions $u_j\in\Np(\Omega)$ such that 
$u_j\geq 1$ on $K$ for all $j=1,2,\dots$\,, and 
\begin{equation}\label{def:p-parabolicity-eq}
	\int_\Omega g_{u_j}^p\,\dmu 
	\to 0
	\quad\text{as }j\to\infty.
\end{equation}
Otherwise, $\Omega$ is said to be \emph{\pp-hyperbolic}.
\end{definition}
In Definition~\ref{def:p-parabolicity}, 
we may as well use 
$u_j\in\Dp(\Omega)$ 
with bounded support such that 
$\chi_K\leq u_j\leq 1$, $j=1,2,\dots$ 
(see, e.g., the proof of Lemma~5.43 in Bj\"orn--Bj\"orn~\cite{BjBj11book}). 
\begin{remark}\label{rem:p-parabolic-subset}
If $\Omega_1\subset\Omega_2$, 
then $\Omega_1$ is \pp-parabolic 
whenever $\Omega_2$ is \pp-parabolic. 
\end{remark}
Holopainen--Shanmugalingam~\cite{HoSh02} 
proposed a definition of \pp-harmonic Green functions 
(i.e., fundamental solutions of the \pp-Laplace operator) 
on metric spaces. 
The functions they defined 
did, however, not share all characteristics with Green functions, 
and therefore they gave them another name; 
they called them \emph{\pp-singular functions}. 
Theorem~3.14 in \cite{HoSh02} asserts that 
if $X$ is locally linearly locally connected 
(see Section~2 in \cite{HoSh02} for the definition), 
then the space $X$ is \pp-hyperbolic 
if and only if 
for every $y\in X$ 
there exists a \pp-singular function 
with singularity at $y$.
\begin{example}\label{rem:Rn-p-parabolic}
The space $\R^n$, $n\geq 1$, is \pp-parabolic 
if and only if $p\geq n$. 
(It follows that all open subsets of $\R^n$ 
are \pp-parabolic for all $p\geq n$; 
see Remark~\ref{rem:p-parabolic-subset}.)

To see this, 
assume that $p\geq n$ and let $K\subset\R^n$ be compact. 
Choose $R$ sufficiently large so that 
$K\subset B:=B(0,R)$. 
Let
\begin{equation}\label{rem:Rn-p-parabolic-eq}
	u_j(x) 
	= \min\biggl\{1,\biggl(1-\frac{\log|x/R|}{j}\biggr)_+\,\biggr\},
	\quad j=1,2,\dots.
\end{equation}
Then $\{u_j\}_{j=1}^\infty$ is a sequence of admissible functions 
for \eqref{def:p-parabolicity-eq}, 
and 
\[
	g_{u_j}
	= (j\,|x|)^{-1}\chi_{B_j\setm B},
	\quad j=1,2,\dots,
\]
where $B_j:=B(0,Re^j)$. 
It follows that 
\[
	\int_{\R^n}g_{u_j}^p\,\dx 
	= C_n\int_R^{Re^j}\frac{r^{n-1}}{(jr)^p}\,\dr 
	= C_n\begin{cases}
			\dfrac{R^{n-p}(1-e^{-j(p-n)})}{(p-n)j^p} & \text{if }p>n, \\
			\,j^{1-p} & \text{if }p=n,
		\end{cases}
\]
and hence 
$\int_{\R^n}g_{u_j}^p\,\dx\to 0$ as $j\to\infty$.

The necessity follows from Theorem~3.14 in 
Holopainen--Shanmugalingam~\cite{HoSh02},
because if we assume that $p<n$ and let $y\in\R^n$, 
then 
\[
	f(x)=|x-y|^{\tfrac{p-n}{p-1}}, 
	\quad x\in\R^n,
\]
is a Green function with singularity at $y$ 
that is \pp-harmonic in $\R^n\setm\{y\}$. 
\end{example}

\medskip
A set can be \pp-parabolic 
if it does not ``grow too much'' towards infinity, 
even though the surrounding space is not \pp-parabolic. 
\begin{example}
Let $n\geq 2$ and 
assume that 
$1<p<n$. 
Let 
\[
	\Omega_f 
	= \{x=(x',\tilde x)\in\R\times\R^{n-1}:0<x'<f(|\tilde x|)\}, 
\]
where 
\[
	f(r) \leq 
	\begin{cases}
		C & \text{if } r < 1, \\
		Cr^q & \text{if } r\geq 1,
	\end{cases}
\]
and $q\leq p-n+1$ 
(note that $q<1$ since $p<n$). 

Let $K\subset\Omega_f$ be compact. 
Choose $R$ sufficiently large so that 
$K\subset B:=B(0,R)$. 
It can be chosen large enough so that 
$|\tilde x|\geq R/2\geq 1$ for all $(x',\tilde x)\in\Omega_f\setm B$. 
This is possible since $q<1$ and $f(r)<Cr^q$. 
Define the sequence of admissible functions $\{u_j\}_{j=1}^\infty$ 
as in \eqref{rem:Rn-p-parabolic-eq}. 
Then 
\begin{align*}
	\int_{\Omega_f}g_{u_j}^p\,\dx 
	&= \int_{\R^{n-1}}\int_0^{f(|\tilde x|)} 
		\frac{\chi_{B_j\setm B}}{(j|x|)^p}\,\dx'\,d\tilde x \\
	&\leq \frac{C_{n-1}}{j^p}\int_{R/2}^{Re^j}\frac{f(r)}{r^p}\,r^{n-2}\,\dr 
	= \frac{C'_{n-1}}{j^p}\int_{R/2}^{Re^j}r^{q-p+n-2}\,\dr 
	=: I_j.
\end{align*}
Since 
\[
	\int_{R/2}^{Re^j} r^{q-p+n-2}\,\dr 
	= \begin{cases}
		j+\log{2} & \text{if }q=p-n+1, \\[0.5em]
		\dfrac{(e^{j(q-p+n-1)}-2^{-(q-p+n-1)})R^{q-p+n-1}}{q-p+n-1} 
			& \text{if }q<p-n+1,
	\end{cases}
\]
it follows that	
$\int_{\Omega_f}g_{u_j}^p\,\dx\leq I_j\to 0$ as $j\to\infty$. 
Thus $\Omega_f$ is \pp-parabolic 
(while $\R^n$ is not \pp-parabolic since $p<n$ in this case).
\end{example}

\section{\texorpdfstring{\boldmath$p\mspace{1mu}$}{p}-harmonic and superharmonic functions}
\label{sec:p-harmonic} 
\emph{The standing assumptions are 
described at the end of Section~\ref{sec:obstacle-problem}.}

\bigskip
There are many equivalent definitions of (super)minimizers 
(or, more accurately, \pp-(super)minimizers) 
in the literature 
(see, e.g., Proposition~3.2 in A.~Bj\"orn~\cite{BjornA06a}). 
\begin{definition}\label{def:min}
We say that a function $u\in\Nploc(\Omega)$ 
is a \emph{superminimizer} in $\Omega$ if  
\begin{equation}\label{def:min-eq}
	\int_{\phi\neq 0}g_u^p\,\dmu 
	\leq \int_{\phi\neq 0}g_{u+\phi}^p\,\dmu
\end{equation}
holds for all nonnegative $\phi\in\Np_0(\Omega)$, 
and a \emph{minimizer} in $\Omega$ if 
\eqref{def:min-eq} holds for all $\phi\in\Np_0(\Omega)$. 
Moreover, a function is \emph{\pp-harmonic} 
if it is a continuous minimizer.
\end{definition}
According to Proposition~3.2 in A.~Bj\"orn~\cite{BjornA06a}, 
it is in fact only necessary to test \eqref{def:min-eq} with 
(all nonnegative and all, respectively) $\phi\in\Lipc(\Omega)$. 

Proposition~3.9 in Hansevi~\cite{Hansevi15} asserts that a function $u$ 
is a superminimizer in $\Omega$ 
if $u$ is a solution of the $\K_{\psi,f}$-obstacle problem.

The following definition makes sense due to 
Theorem~4.4 in Hansevi~\cite{Hansevi15}. 
Because Proposition~2.7 in Bj\"orn--Bj\"orn~\cite{BjBj12a} 
asserts that $\Dp_0(\Omega)=\Np_0(\Omega)$ if $\Omega$ is bounded, 
it is a generalization of 
Definition~8.31 in Bj\"orn--Bj\"orn~\cite{BjBj11book} 
to Dirichlet functions and to unbounded sets. 
\begin{definition}\label{def:ext}
Let $V\subset X$ be a nonempty open set with $\Cp(X\setm V)>0$. 
The \emph{\pp-harmonic extension} 
$\Hp_V f$ of $f\in\Dp(V)$ to $V$ is the continuous solution 
of the $\K_{-\infty,f}(V)$-obstacle problem. 
When $V=\Omega$ we usually write $\Hp f$ instead of $\Hp_\Omega f$.
\end{definition}
If $f$ is defined outside $V$, 
then we sometimes consider $\Hp_V f$ to be equal to $f$ 
in some set outside $V$ where $f$ is defined. 

A Lipschitz function $f$ on $\bdy V$ can be extended to 
a Lipschitz function $\bar f$ on $\overline V$ 
(see, e.g., Theorem~6.2 in Heinonen~\cite{Heinonen01}), 
and $\bar f\in\Np(\overline V)$ 
if $V$ is bounded. 
The comparison principle (Lemma~4.7 in Hansevi~\cite{Hansevi15}) 
implies that 
$\Hp_V \bar f$ does not depend on the 
particular choice of extension $\bar f$. 
We can therefore define the \pp-harmonic extension for 
Lipschitz functions on the boundary by 
$\Hp_V f:=\Hp_V\bar f$ 
if $V$ is bounded. 
\begin{proposition}\label{prop:Hf-lim}
If\/ $\{f_j\}_{j=1}^\infty$ is a sequence of functions in $\Dp(\Omega)$ 
that is decreasing q.e.\ in\/ $\Omega$ to $f\in\Dp(\Omega)$ 
and\/ $\|g_{f_j-f}\|_{\Lp(\Omega)}\to 0$ as $j\to\infty$\textup{,} 
then $\Hp f_j$ decreases to $\Hp f$ locally uniformly in\/ $\Omega$.
\end{proposition}
\begin{proof}
By the comparison principle (Lemma~4.7 in Hansevi~\cite{Hansevi15}), 
it follows that  
$\Hp f_j\geq\Hp f_{j+1}\geq\Hp f$ in $\Omega$ for all $j=1,2,\dots$\,. 
Since $\Hp f_j$ and $\Hp f$ are the continuous solutions of the 
$\K_{f_j,\Hp f}$- and $\K_{f,\Hp f}$-obstacle problems, 
respectively, 
it follows from Theorem~\ref{thm:obst-lim} 
that $\Hp f_j$ decreases to $\Hp f$ q.e.\ in $\Omega$ 
as $j\to\infty$. 

Because $\Hp f$ is continuous, 
and therefore locally bounded, 
Proposition~5.1 in Shanmugalingam~\cite{Shanmugalingam03} 
implies that $\Hp f_j\to\Hp f$ locally uniformly in $\Omega$ as $j\to\infty$.
\end{proof}

In order to define Perron solutions, 
we need superharmonic functions. 
We follow Kinnunen--Martio~\cite{KiMa02}, 
however, we use a slightly different, 
nevertheless equivalent, 
definition (see, e.g., Proposition~9.26 in Bj\"orn--Bj\"orn~\cite{BjBj11book}).
\begin{definition}\label{def:superharm}
We say that a function $u\colon\Omega\to(-\infty,\infty]$ 
is \emph{superharmonic} in $\Omega$ if 
\begin{enumerate}
\item $u$ is lower semicontinuous;
\item $u$ is not identically $\infty$ 
		in any component of $\Omega$;
\item \label{def:superharm-c} for every nonempty open set 
		$V'\Subset \Omega$ and all 
		$v\in\Lip(\bdy V')$, 
		we have $\Hp_{V'}v\leq u$ in $V'$ 
		whenever $v\leq u$ on $\bdy V'$. 
\end{enumerate}
A function $u\colon\Omega\to[-\infty,\infty)$ 
is \emph{subharmonic} in $\Omega$ 
if the function $-u$ is superharmonic.
\end{definition}

\section{Perron solutions}
\label{sec:Perron} 
\emph{The standing assumptions are 
described at the end of Section~\ref{sec:obstacle-problem}. 
We make the convention from now on that 
the point at infinity\textup{,} 
$\infty$\textup{,} 
belongs to the boundary $\bdy\Omega$ 
if\/ $\Omega$ is unbounded. 
Topological notions should therefore be 
understood with respect to 
the one-point compactification $X^*:=X\cup\{\infty\}$.}

\bigskip
\begin{definition}\label{def:Perron}
Given a function $f\colon\bdy\Omega\to\eR$, 
we let $\U_f(\Omega)$ be the set of all 
superharmonic functions $u$ in $\Omega$ 
that are bounded below 
and such that 
\[
	\liminf_{\Omega\ni y\to x}u(y) 
	\geq f(x)
	\quad\textup{for all }x\in\bdy\Omega. 
\]
Then the 
\emph{upper Perron solution} of $f$ is defined by 
\[
	\uPp_\Omega f(x) 
	= \inf_{u\in\U_f(\Omega)}u(x),
	\quad x\in\Omega.
\]
Similarly, 
we let $\LL_f(\Omega)$ be the set of all 
subharmonic functions $v$ in $\Omega$ 
that are bounded above 
and such that 
\[
	\limsup_{\Omega\ni y\to x}v(y) 
	\leq f(x)
	\quad\textup{for all }x\in\bdy\Omega, 
\]
and define the 
\emph{lower Perron solution} of $f$ by 
\[
	\lPp_\Omega f(x) 
	= \sup_{v\in\LL_f(\Omega)}v(x),
	\quad x\in\Omega.
\]
If $\uPp_\Omega f=\lPp_\Omega f$, 
then we let $\Pp_\Omega f:=\uPp_\Omega f$. 
Moreover, if $\Pp_\Omega f$ is real-valued, 
then $f$ is said to be \emph{resolutive} 
(with respect to $\Omega$). 
We often write $\Pp f$ instead of $\Pp_\Omega f$.
\end{definition}
Immediate consequences of the above definition are that 
$\lPp f=-\uPp(-f)$ and that $\uPp f\leq\uPp h$ if $f\leq h$. 
It also follows that 
$\uPp f=\lim_{k\to\infty}\uPp\max\{f,-k\}$.

In each component of $\Omega$, 
$\uPp f$ is either \pp-harmonic or identically $\pm\infty$, 
see, e.g.,  
Bj\"orn--Bj\"orn~\cite{BjBj11book} 
(their proof applies also to unbounded $\Omega$). 
Thus Perron solutions are reasonable candidates 
for solutions of the Dirichlet problem. 

The following theorem extends the comparison principle, 
which is fundamental for the nonlinear potential theory 
of superharmonic functions, 
and also plays an important role for the Perron method. 
\begin{theorem}\label{thm:comparison}
If $u$ is superharmonic and $v$ is subharmonic in\/ $\Omega$\textup{,} 
then $v\leq u$ in\/ $\Omega$ whenever 
\begin{equation} \label{thm:comparison-ineq}
	\infty 
	\neq \limsup_{\Omega\ni y\to x}v(y)
	\leq \liminf_{\Omega\ni y\to x}u(y)
	\neq -\infty
\end{equation}
for all $x\in\bdy\Omega$ 
\textup{(}i.e., also for $x=\infty$ if\/ $\Omega$ is unbounded\,\textup{)}. 
\end{theorem}
\begin{corollary}\label{cor:comparison}
If $f\colon\bdy\Omega\to\eR$, 
then $\lPp f\leq\uPp f$.
\end{corollary}
\begin{proof}[Proof of Theorem~\ref{thm:comparison}]
Fix $\eps>0$. 
For each $x\in\bdy\Omega$, 
it follows from \eqref{thm:comparison-ineq} that 
\[
	\liminf_{\Omega\ni y\to x}(u(y)-v(y)) 
	\geq \liminf_{\Omega\ni y\to x}u(y) - \limsup_{\Omega\ni y\to x}v(y)
	\geq 0, 
\]
and hence there is an open set $U_x\subset X^*$ such that $x\in U_x$ and 
\[
	u-v 
	\geq -\eps
	\quad\textup{in }U_x\cap\Omega. 
\]

Let $\Omega_1,\Omega_2,\dots$ be open sets such that 
$\Omega_1\Subset\Omega_2\Subset\cdots\Subset\Omega=\bigcup_{k=1}^\infty\Omega_k$. 
Then 
\[
	\clOm 
	\subset \bigcup_{k=1}^\infty\Omega_k
	\;\cup\bigcup_{x\in\bdy\Omega}U_x.
\]
Since $\clOm$ is compact 
(with respect to the topology of $X^*$), 
there exist integers 
$k>1/\eps$ and $N$ 
such that 
\[
	\clOm
	\subset\Omega_k\cup U_{x_1}\cup\cdots\cup U_{x_N}.
\]
It follows that $v\leq u+\eps$ on $\bdy\Omega_k$. 
Since $v$ is upper semicontinuous 
(and does not take the value $\infty$), 
it follows that there is a decreasing sequence 
$\{\phi_j\}_{j=1}^\infty\subset\Lip(\clOm_k)$ 
such that $\phi_j\to v$ on \smash{$\clOm_k$} 
as $j\to\infty$ 
(see, e.g., Proposition~1.12 in Bj\"orn--Bj\"orn~\cite{BjBj11book}).  
 
Since $u+\eps$ is lower semicontinuous, 
the compactness of $\bdy\Omega_k$ shows that 
there exists an integer $M$ such that $\phi_M\leq u+\eps$ 
on $\bdy\Omega_k$, 
and, by \ref{def:superharm-c} in Definition~\ref{def:superharm}, 
also that $\Hp_{\Omega_k}\phi_M\leq u+\eps$ in $\Omega_k$. 
Similarly, $v\leq\Hp_{\Omega_k}\phi_M$, 
and thus $v\leq u+\eps$ in $\Omega_k$.
Letting $\eps\to 0$ (and hence letting $k\to\infty$) 
implies that $v\leq u$ in $\Omega$. 
\end{proof}

\section{Resolutivity of functions on \texorpdfstring{\boldmath$\bdy\Omega$}{the boundary}}
\label{sec:resolutivity} 
\emph{In addition to the standing assumptions 
described at the end of Section~\ref{sec:obstacle-problem}\textup{,} 
we assume that\/ $\Omega$ is \pp-parabolic 
if\/ $\Omega$ is unbounded \textup{(}see Definition~\ref{def:p-parabolicity}\textup{)}. 
For the convention about the point at infinity\textup{,} 
see the beginning of Section~\ref{sec:Perron}.}

\bigskip
When Bj\"orn--Bj\"orn--Shanmugalingam~\cite{BjBjSh13a} 
extended the Perron method to the Mazurkiewicz boundary 
of bounded domains that are finitely connected at the boundary, 
they introduced a new capacity, 
$\bCpO$, 
adapted to the topology that connects the domain to its Mazurkiewicz boundary. 
They also used the new capacity to define 
$\bCpO$-quasicontinuous functions. 
By using $\bCpO$, 
which is smaller than the usual Sobolev capacity 
(see the appendix of \cite{BjBjSh13a}), 
we allow for perturbations on larger sets 
and we obtain resolutivity for more functions.
\begin{definition}\label{def:new-capacity}
The $\bCpO$-\emph{capacity} 
of a set $E\subset\clOm$ is the number 
\[
	\bCp(E;\Omega) 
	:= \inf_{u\in\calV_E}\|u\|_{\Np(\Omega)}^p
\]
where $\calV_E$ is the family of all functions 
$u\in\Np(\Omega)$ that satisfy both 
$u(x)\geq 1$ for all $x\in E\cap\Omega$ and 
\begin{equation}\label{def:new-capacity-ineq}
	\liminf_{\Omega\ni y\to x}u(y) 
	\geq 1 
	\quad\textup{for all }x\in E\cap\bdy\Omega. 
\end{equation}

When a property holds for all points 
except for points in a set of 
$\bCpO$-capacity zero, 
it is said to hold \emph{$\bCpO$-quasieverywhere} 
(or $\bCpO$-\emph{q.e.} for short). 
\end{definition}
If $E\subset\Omega$, 
then condition \eqref{def:new-capacity-ineq} 
becomes empty and 
$\bCp(E;\Omega)=\Cp(E;\Omega)$. 

The capacity $\bCpO$ shares several properties 
with the Sobolev capacity, 
e.g., monotonicity and countable subadditivity. 
Moreover, 
$\bCpO$ is an outer capacity, i.e., 
if $E\subset\clOm$\textup{,} 
then  
\[
	\bCp(E;\Omega) 
	= \inf_{\substack{G\supset E \\ G\text{ relatively open in }\clOm}}
	\bCp(G;\Omega). 
\]
These results are proved in Bj\"orn--Bj\"orn--Shanmugalingam~\cite{BjBjSh13a} 
(a slightly modified version of their proof that $\bCpO$ is outer 
is valid in our setting as well). 

To prove Theorem~\ref{thm:Dp-resolutive}, 
we need the following version of 
Lemma~5.3 in Bj\"orn--Bj\"orn--Shanmugalingam~\cite{BjBjSh03b}.
\begin{lemma}\label{lem:5.3-BjBjSh03b}
Assume that\/ $\{U_k\}_{k=1}^\infty$ is a decreasing sequence 
of relatively open subsets of\/ $\clOm$ with 
$\bCp(U_k;\Omega)<2^{-kp}$. 
Then there exists a sequence of nonnegative functions 
$\{\psi_j\}_{j=1}^\infty$ 
that decreases to zero q.e.\ in $\Omega$\textup{,} 
such that\/ 
$\|\psi_j\|_{\Np(\Omega)}<2^{-j}$ and  
$\psi_j\geq k-j$ in $U_k\cap\Omega$.
\end{lemma}
\begin{proof}
For each $k=1,2,\dots$\,, 
there exists a nonnegative function $u_k$ such that 
$u_k=1$ in $U_k\cap\Omega$ and 
$\|u_k\|_{\Np(\Omega)}<2^{-k}$ 
because $\bCp(U_k;\Omega)<2^{-kp}$. 
Letting 
\[
	\psi_j 
	= \sum_{k=j+1}^\infty u_k, 
	\quad j=1,2,\dots,
\]
yields a decreasing sequence of nonnegative functions 
such that $\|\psi_j\|_{\Np(\Omega)}<2^{-j}$ and  
$\psi_j\geq k-j$ in $U_k\cap\Omega$. 
Corollary~3.9 in Shanmugalingam~\cite{Shanmugalingam00} 
implies the existence of a subsequence of $\{\psi_j\}_{j=1}^\infty$ 
that converges to zero q.e.\ in $\Omega$, 
and since $\{\psi_j\}_{j=1}^\infty$ is nonnegative and decreasing, 
this shows that $\{\psi_j\}_{j=1}^\infty$ 
decreases to zero q.e.\ in $\Omega$.
\end{proof}
\begin{definition}\label{def:new-quasicont} 
Let $f$ be an extended real-valued function defined on $\clOmX$. 
We say that $f$ is \emph{$\bCpO$-quasicontinuous on} $\clOmX$ 
if for every $\eps>0$ 
there is a relatively open subset $U$ of $\clOmX$ 
with $\bCp(U;\Omega)<\eps$ 
such that the restriction of $f$ to $(\clOmX)\setm U$ 
is continuous and real-valued. 
\end{definition}
Since the $\bCpO$-capacity is smaller than the Sobolev capacity 
(which is used to define quasicontinuity), 
it follows that quasicontinuous functions are also 
$\bCpO$-quasicontinuous. 
\begin{proposition}\label{prop:Dp0-new-quasicont}
If $f\colon\clOmX\to\eR$ is a function such that 
$f=0$ q.e.\ on $\bdyOmX$ and 
$f|_\Omega\in\Dp_0(\Omega)$, 
then $f$ is $\bCpO$-quasicontinuous on $\clOmX$.
\end{proposition}
\begin{proof}
Extend $f$ to $X$ by letting $f$ be equal to zero outside $\clOm$ 
so that $f\in\Dp(X)$. 
Then $f\in\Nploc(X)$ by Proposition~4.14 in Bj\"orn--Bj\"orn~\cite{BjBj11book}, 
and hence Theorem~1.1 in Bj\"orn--Bj\"orn--Shanmugalingam~\cite{BjBjSh08} 
asserts that $f$ is quasicontinuous on $X$, 
and therefore $\bCpO$-quasicontinuous on $\clOmX$. 
\end{proof}

The following is the main result of this paper.
\begin{theorem} \label{thm:Dp-resolutive} 
Assume that $f\colon\clOm\to\eR$ is $\bCpO$-quasicontinuous on $\clOmX$ 
and such that $f|_\Omega\in\Dp(\Omega)$\textup{,} 
which in particular hold if $f\in\Dp(X)$. 
Then $f$ is resolutive with respect to $\Omega$ and $\Pp f=\Hp f$.
\end{theorem}
To see that \pp-parabolicity is needed in 
Theorem~\ref{thm:Dp-resolutive} if $\Omega$ is unbounded, 
let $n>p$ and let $\Omega=\R^n\setm\itoverline B$, 
where $B$ is the open unit ball centered at the origin.
Then $\Omega$ is \pp-hyperbolic. 
Furthermore, let 
\[
	f(x)=|x|^{\tfrac{p-n}{p-1}}, 
	\quad x\in\clOm.
\]
Then $f$ satisfies the hypothesis of Theorem~\ref{thm:Dp-resolutive}. 
Because $f\equiv 1$ on $\bdy B$ 
and the \pp-harmonic extension does not consider the point at infinity, 
it is clear that $\Hp f\equiv 1$. 
However, $\Pp f\equiv f$, 
since $f$ is in fact \pp-harmonic 
(it is easy to verify that $f$ 
is a solution of the \pp-Laplace equation 
\eqref{sec:intro-p-Laplace-eq}) 
and continuous on $\clOm$, 
and hence $f\in\U_f(\Omega)$ and $f\in\LL_f(\Omega)$, 
which implies that 
$f\leq\lPp f\leq\uPp f\leq f$.
\begin{proof}[Proof of Theorem~\ref{thm:Dp-resolutive}]
Suppose that $\Omega$ is unbounded and \pp-parabolic. 
Let $\{K_j\}_{j=1}^\infty$ 
be an increasing sequence of compact sets such that 
$K_1\Subset K_2\Subset\cdots\Subset\Omega=\bigcup_{j=1}^\infty K_j$ 
and let $x_0\in X$. 
For each $j=1,2,\dots$\,, 
we can find a function 
$u_j\in\Dp(\Omega)$ such that $\chi_{K_j}\leq u_j\leq 1$, 
$u_j=0$ in $\Omega\setm B_j$ for some ball $B_j\supset K_j$ 
centered at $x_0$, 
and 
\begin{equation} \label{thm:Dp-resolutive-g_u_j}
	\|g_{u_j}\|_{\Lp(\Omega)}<2^{-j}.
\end{equation}
Let 
\begin{equation} \label{thm:Dp-resolutive-xi_j}
	\xi_j = \sum_{k=j}^\infty(1-u_k),
	\quad j=1,2,\dots. 
\end{equation}
Then $\xi_j\geq 0$ and 
\begin{equation}\label{thm:Dp-resolutive-g_xi_j}
	\|g_{\xi_j}\|_{\Lp(\Omega)}
	\leq \sum_{k=j}^\infty\|g_{u_k}\|_{\Lp(\Omega)}
	< \sum_{k=j}^\infty 2^{-k} 
	= 2^{1-j}. 
\end{equation}
Let $\Omega_j=\bigcup_{n=1}^{\,j}B_n\cap\Omega$, $j=1,2,\dots$\,. 
Then 
$\Omega_1\subset\Omega_2\subset\cdots\subset\Omega=\bigcup_{j=1}^\infty\Omega_j$. 
Since 
$u_j=0$ in $\Omega\setm\Omega_j$, 
it is easy to see that
\begin{equation} \label{thm:Dp-resolutive-limit}
	\lim_{\Omega\ni y\to\infty}\xi_j(y) 
	= \infty 
	\quad\text{for all }j=1,2,\dots.
\end{equation}
Furthermore, 
since 
$\{\xi_j\}_{j=1}^\infty$ is decreasing 
and $\xi_j=0$ on $K_j$ 
for each $j=1,2,\dots$\,, 
it follows that 
$\{\xi_j\}_{j=1}^\infty$ 
decreases to zero in $\Omega$. 

On the other hand, if $\Omega$ is bounded, 
then we let $\xi_j\equiv 0$ in $\Omega$, $j=1,2,\dots$\,. 

The \pp-harmonic extension 
$\Hp f$ is 
$\bCpO$-quasicontinuous on $\clOmX$ 
(when we consider $\Hp f$ to be equal to $f$ on $\bdy\Omega$), 
since Proposition~\ref{prop:Dp0-new-quasicont} 
asserts that $\Hp f-f$ is $\bCpO$-quasicontinuous on $\clOmX$ as 
$(\Hp f-f)|_\Omega\in\Dp_0(\Omega)$. 
We can therefore find a decreasing sequence $\{U_k\}_{k=1}^\infty$ 
of relatively open subsets of $\clOmX$ 
with $\bCp(U_k;\Omega)<2^{-kp}$ and such that 
the restriction of $\Hp f$ to $(\clOmX)\setm U_k$ 
is continuous. 

Now we derive that 
$\uPp f\leq\Hp f$ q.e.\ in $\Omega$ 
if $f$ is bounded from below. 
Without loss of generality, we may as well assume that $f\geq 0$. 
Then the comparison principle (Lemma~4.7 in Hansevi~\cite{Hansevi15}) 
implies that $\Hp f\geq 0$ in $\Omega$. 

Consider the sequence of nonnegative functions 
$\{\psi_j\}_{j=1}^\infty$ 
given by Lemma~\ref{lem:5.3-BjBjSh03b}, 
and define $h_j\colon\Omega\to[0,\infty]$ by letting 
\[
	h_j 
	= \Hp f + \xi_j + \psi_j, 
	\quad j=1,2,\dots. 
\]
Then $h_j\in\Dp(\Omega)$ and 
$\{h_j\}_{j=1}^\infty$ 
decreases to $\Hp f$ q.e.\ in $\Omega$. 

Let $\phi_j$ 
be the lsc-regularized solution of the 
$\K_{h_j,h_j}$-obstacle problem, $j=1,2,\dots$\,. 
By \eqref{thm:Dp-resolutive-g_xi_j} and 
Lemma~\ref{lem:5.3-BjBjSh03b}, 
\[
	\|g_{h_j-\Hp f}\|_{\Lp(\Omega)} 
	\leq \|g_{\xi_j}\|_{\Lp(\Omega)}
		+ \|g_{\psi_j}\|_{\Lp(\Omega)}
	< 2^{1-j}
		+ 2^{-j}
	\to 0
	\quad\text{as }j\to\infty,
\]
and as $\Hp f$ is a solution of the 
$\K_{\Hp f,\Hp f}$-obstacle problem, 
it follows from Theorem~\ref{thm:obst-lim} that 
$\{\phi_{j}\}_{j=1}^\infty$ decreases to $\Hp f$ q.e.\ in $\Omega$. 
This will be used later in the proof.

Next we show that 
\begin{equation} \label{thm:Dp-resolutive-liminf}
	\liminf_{\Omega\ni y\to x}\phi_j(y) 
	\geq f(x)
	\quad\textup{for all }x\in\bdy\Omega.
\end{equation}
Fix a positive integer $m$ and let $\eps=1/m$. 
By Lemma~\ref{lem:5.3-BjBjSh03b}, 
\begin{equation} \label{thm:Dp-resolutive-ineq-1}
	h_j(y) 
	\geq \psi_j(y) 
	\geq m
	\quad\textup{for all }y\in U_{m+j}\cap\Omega.
\end{equation}
Let $x\in\bdyOmX$. 
If $x\notin U_{m+j}$, 
then as the restriction of $\Hp f$ to 
$(\clOmX)\setm U_{m+j}$ 
is continuous, 
there is a relative neighborhood 
$V_x\subset\clOmX$ of $x$ such that 
\begin{equation} \label{thm:Dp-resolutive-ineq-2}
	h_j(y) 
	\geq \Hp f(y)
	\geq \Hp f(x)-\eps 
	= f(x)-\eps
	\quad\textup{for all }y\in (V_x\cap\Omega)\setm U_{m+j}.
\end{equation}
By combining
\eqref{thm:Dp-resolutive-ineq-1} 
and 
\eqref{thm:Dp-resolutive-ineq-2}, 
we see that for $x\in(\bdyOmX)\setm U_{m+j}$, 
\begin{equation} \label{thm:Dp-resolutive-ineq-3}
	h_j(y) 
	\geq \min\{f(x)-\eps,m\}
	\quad\textup{for all }y\in V_x\cap\Omega. 
\end{equation}
On the other hand, 
if $x\in U_{m+j}$, 
then we let $V_x=U_{m+j}$, 
and see that \eqref{thm:Dp-resolutive-ineq-3} 
holds also in this case due to \eqref{thm:Dp-resolutive-ineq-1}.
Because $\phi_j\geq h_j$ q.e.\ in $\Omega$ 
and $\phi_j$ is lsc-regularized, 
it follows that 
\[
	\phi_j(y) 
	\geq \min\{f(x)-\eps,m\}
	\quad\textup{for all }y\in V_x\cap\Omega, 
\]
and hence 
\[ 
	\liminf_{\Omega\ni y\to x}\phi_j(y) 
	\geq \min\{f(x)-\eps,m\}.
\]
Letting $m\to\infty$ (and thus letting $\eps\to 0$) 
establishes that  
\[
	\liminf_{\Omega\ni y\to x}\phi_j(y) 
	\geq f(x)
	\quad\textup{for all }x\in\bdyOmX.
\]
Finally, if $\Omega$ is unbounded, 
then $\phi_j\geq h_j$ q.e.\ in $\Omega$ 
and $h_j\geq\xi_j$ everywhere in $\Omega$. 
From the lsc-regularity of $\phi_j$ and 
\eqref{thm:Dp-resolutive-limit}, 
it follows that 
\[
	\liminf_{\Omega\ni y\to\infty}\phi_j(y) 
	\geq 
	\lim_{\Omega\ni y\to\infty}\xi_j(y) 
	= 
	\infty,
\]
and hence we have shown that \eqref{thm:Dp-resolutive-liminf} holds. 

Since $\phi_j$ is an lsc-regularized superminimizer, 
Proposition~7.4 in Kinnunen--Martio~\cite{KiMa02} asserts that  
$\phi_j$ is superharmonic. 
As $\phi_j$ is bounded from below 
and \eqref{thm:Dp-resolutive-liminf} holds, 
it follows that $\phi_j\in\U_f(\Omega)$, 
and hence we know that $\uPp f\leq\phi_j$, 
$j=1,2,\dots$\,. 
Because $h_j\in\Dp(\Omega)$ and 
$\{h_j\}_{j=1}^\infty$ 
decreases to $\Hp f$ q.e.\ in $\Omega$, 
$\|g_{h_j-\Hp f}\|_{\Lp(\Omega)}\to 0$ as $j\to\infty$, 
and $\Hp f$ is a solution of the 
$\K_{\Hp f,\Hp f}$-obstacle problem, 
it follows from Theorem~\ref{thm:obst-lim} that 
$\{\phi_{j}\}_{j=1}^\infty$ decreases to $\Hp f$ q.e.\ in $\Omega$. 
We therefore conclude that $\uPp f\leq\Hp f$ q.e.\ in $\Omega$ 
(provided that $f$ is bounded from below).

Now we remove the extra assumption 
of $f$ being bounded from below, 
and let $f_k=\max\{f,-k\}$, $k=1,2,\dots$\,. 
Then $\{f_k\}_{k=1}^\infty$ is decreasing to $f$. 
Proposition~4.14 in Bj\"orn--Bj\"orn~\cite{BjBj11book} 
implies that $f\in\Lploc(\Omega)$. 
Hence $\mu(\{x\in\Omega:|f(x)|=\infty\})=0$, 
and therefore 
$\chi_{\{x\in\Omega\,:\,f(x)<-k\}}\to 0$ a.e.\ in $\Omega$ as $k\to\infty$.
Since 
\[
	g_{f_k-f} 
	= g_{\max\{0,-f-k\}}
	= g_f\chi_{\{x\in\Omega\,:\,f(x)<-k\}}
	\quad\textup{a.e.\ in }\Omega, 
\]
implies that 
$g_{f_k-f}\to 0$ a.e.\ in $\Omega$ as $k\to\infty$, 
and because $g_f\in\Lp(\Omega)$ and 
\[
	g_{f_k-f} 
	\leq g_{f_k}+g_f 
	\leq 2g_f
	\quad\textup{a.e.\ in }\Omega,
\] 
it follows by dominated convergence that  
$g_{f_k-f}\to 0$ in $\Lp(\Omega)$ as $k\to\infty$.
Thus Proposition~\ref{prop:Hf-lim} asserts that 
\[
	\Hp f_k\to\Hp f
	\quad\textup{in }\Omega\textup{ as }k\to\infty. 
\]
Since $f_k$ is bounded from below, 
it follows that  
\[
	\uPp f 
	= \lim_{k\to\infty} \uPp f_k
	\leq \lim_{k\to\infty}\Hp f_k 
	= \Hp f
	\quad\textup{q.e.\ in }\Omega.
\]
As both $\uPp f$ and $\Hp f$ are continuous, 
we conclude that $\uPp f\leq\Hp f$
everywhere in $\Omega$.
By Corollary~\ref{cor:comparison}, 
it follows that 
\[ 
	\uPp f 
	\leq \Hp f 
	= -\Hp(-f) 
	\leq -\uPp(-f) 
	= \lPp f 
	\leq \uPp f
	\quad\textup{in }\Omega,
\]
which implies that $f$ is resolutive and that $\Pp f=\Hp f$. 
\end{proof}

Perron solutions are invariant under perturbation of the function on a set of capacity zero.
\begin{theorem} \label{thm:Dp-invariance} 
Assume that $f\colon\clOm\to\eR$ is $\bCpO$-quasicontinuous on $\clOmX$ 
and such that 
$f|_\Omega\in\Dp(\Omega)$\textup{,} 
which in particular hold if $f\in\Dp(X)$. 
Assume also that $h\colon\bdy\Omega\to\eR$ is zero 
$\bCpO$-q.e.\ on $\bdyOmX$. 
Then $f+h$ is resolutive with respect to $\Omega$ and $\Pp(f+h)=\Pp f$.
\end{theorem}
\begin{proof}
Extend $h$ by zero in $\Omega$ and let 
$E=\{x\in\clOm:h(x)\neq 0\}$. 
Since $\bCpO$ is an outer capacity, 
it follows that given $\eps>0$, 
we can find a relatively open subset $U$ of $\clOmX$ 
with $\bCp(U;\Omega)<\eps$ and such that $E\subset U$, 
and hence $h$ is $\bCpO$-quasicontinuous on $\clOmX$. 
The subadditivity of the 
$\bCpO$-capacity implies that this is true also for 
$f+h$. 

Since $f+h=f$ in $\Omega$ and $f|_\Omega\in\Dp(\Omega)$, 
we know that $\Hp(f+h)=\Hp f$. 
We complete the proof by applying Theorem~\ref{thm:Dp-resolutive} 
to both $f$ and $f+h$, 
which shows that $f+h$ is resolutive and that 
\[
	\Pp(f+h) 
	= \Hp(f+h) 
	= \Hp f
	= \Pp f.
	\qedhere
\]
\end{proof}

The following uniqueness result is a direct consequence 
of Theorem~\ref{thm:Dp-invariance}.
\begin{corollary} \label{cor:Dp-invariance}
Assume that $u$ is bounded and \pp-harmonic in $\Omega$. 
Assume also that $f\colon\clOm\to\eR$ is $\bCpO$-quasicontinuous on $\clOmX$ 
and such that $f|_\Omega\in\Dp(\Omega)$. 
Then $u=\Pp f$ in $\Omega$ 
whenever there exists a set 
$E\subset\bdy\Omega$ with $\bCp(E\setm\{\infty\};\Omega)=0$ 
such that 
\[
	\lim_{\Omega\ni y\to x}u(y)=f(x)
	\quad\text{for all }x\in\bdy\Omega\setm E.
\]
\end{corollary}
\begin{proof}
Since $\bCp(E\setm\{\infty\};\Omega)=0$, 
Theorem~\ref{thm:Dp-invariance} applies to 
$f$ and $h:=\infty\chi_E$ (and clearly also to $f$ and $-h$), 
and because $u\in\U_{f-h}(\Omega)$ 
and $u\in\LL_{f+h}(\Omega)$ 
(since $u$ is bounded), 
it follows that 
\[
	u 
	\leq \lPp(f+h) 
	= \Pp(f+h)
	= \Pp f 
	= \Pp(f-h) 
	= \uPp(f-h) 
	\leq u
	\quad\textup{in }\Omega.
	\qedhere
\]
\end{proof}

The obtained resolutivity results can now be extended to continuous functions. 
Bj\"orn--Bj\"orn--Shanmugalingam~\cite{BjBjSh03b},\cite{BjBjSh13a} 
proved the following result for bounded domains. 
\begin{theorem} \label{thm:cont-resolutive}
If $f\in C(\bdy\Omega)$ and $h\colon\bdy\Omega\to\eR$ is zero 
$\bCpO$-q.e.\ on $\bdyOmX$\textup{,} 
then $f$ and $f+h$ are resolutive with respect to\/ $\Omega$ and 
$\Pp(f+h)=\Pp f$.
\end{theorem}
\begin{proof}
We start by choosing a point $x_0\in\bdy\Omega$. 
If $\Omega$ is unbounded, 
then we let $x_0=\infty$. 
Let $\alpha=f(x_0)\in\R$ and 
let $j$ be a positive integer. 
Since $f\in C(\bdy\Omega)$, 
there exists a compact set $K_j\subset X$ such that 
$|f(x)-\alpha|<1/3j$ for all $x\in\bdy\Omega\setm K_j$. 
Let 
\[
	K'_j=\{x\in X:\dist(x,K_j)\leq 1\}.
\] 
We can find functions $\phi_j\in\Lipc(X)$ such that 
$|\phi_j-f|\leq 1/3j$ on $\bdy\Omega\cap K'_j$. 
Let $f_j=(\phi_j-\alpha)\eta_j+\alpha$, where 
\[
	\eta_j(x) 
	:= \begin{cases}
		1, & x\in K_j, \\
		1-\dist(x,K_j), & x\in K'_j\setm K_j, \\
		0, & x\in X\setm K'_j.
	\end{cases}
\]
Since $f_j$ is Lipschitz on $X$ and $f_j=\alpha$ outside $K'_j$, 
it follows that $f_j\in\Dp(X)$. 

Let $x\in\bdy\Omega$. 
Then $|f_j(x)-f(x)|\leq 1/3j$ whenever $x\notin K'_j\setm K_j$. 
Otherwise it follows that 
\begin{align*}
	|f_j(x)-f(x)| 
	&= |(\phi_j(x)-\alpha)\eta_j(x)+\alpha-f(x)| 
	\leq |\phi_j(x)-\alpha)|+|\alpha-f(x)| \\
	&\leq |\phi_j(x)-f(x)|+2|f(x)-\alpha| 
	< \frac{1}{j}, 
\end{align*}
and thus we know that $f-1/j \leq f_j\leq f+1/j$ on $\bdy\Omega$. 
It follows directly from Definition~\ref{def:Perron} that 
$\lPp f-1/j\leq \lPp f_j\leq\lPp f+1/j$, 
and we also get corresponding inequalities for 
$\uPp f_j$, $\lPp(f_j+h)$, and $\uPp(f_j+h)$. 

Theorem~\ref{thm:Dp-invariance} asserts that $f_j$ and $f_j+h$ 
are resolutive and that 
$\Pp(f_j+h)=\Pp f_j$. 
It follows that 
\begin{equation}\label{thm:cont-resolutive-ineq1}
	\uPp f-\frac{1}{j} 
	\leq \uPp f_j
	= \lPp f_j
	\leq \lPp f+\frac{1}{j}. 
\end{equation}
Applying Corollary~\ref{cor:comparison} to 
\eqref{thm:cont-resolutive-ineq1} yields 
$0\leq\uPp f-\lPp f\leq 2/j$. 
Letting $j\to\infty$ shows that $f$ is resolutive. 
Similarly, we can see that also $f+h$ is resolutive.

Finally, we have 
\begin{equation}\label{thm:cont-resolutive-ineq2}
	\Pp(f+h)-\Pp f
	= \uPp(f+h)-\lPp f 
	\leq \uPp(f_j+h)+\frac{1}{j}-\biggl(\lPp f_j-\frac{1}{j}\biggr) 
	= \frac{2}{j}. 
\end{equation}
Interchanging $\uPp(f+h)$ and $\lPp f$ with 
$\lPp(f+h)$ and $\uPp f$, respectively, 
in \eqref{thm:cont-resolutive-ineq2} 
yields $\Pp(f+h)-\Pp f\geq -2/j$, 
and hence 
$|\Pp(f+h)-\Pp f|<2/j$. 
Letting $j\to\infty$ shows that $\Pp(f+h)=\Pp f$.
\end{proof}

We conclude this paper with 
the following uniqueness result, 
corresponding to Corollary~\ref{cor:Dp-invariance}, 
that follows directly from Theorem~\ref{thm:cont-resolutive}. 
The proof is identical to the proof of 
Corollary~\ref{cor:Dp-invariance}, 
except for applying Theorem~\ref{thm:cont-resolutive} 
(instead of Theorem~\ref{thm:Dp-invariance}).
\begin{corollary} \label{cor:cont-resolutive-invariance}
Assume that $u$ is bounded and \pp-harmonic in $\Omega$. 
If $f\in C(\bdy\Omega)$ 
and there is a set 
$E\subset\bdy\Omega$ with $\bCp(E\setm\{\infty\};\Omega)=0$ such that 
\[
	\lim_{\Omega\ni y\to x}u(y)=f(x)
	\quad\text{for all }x\in\bdy\Omega\setm E, 
\]
then $u=\Pp f$ in $\Omega$.
\end{corollary}
%

\newcommand{\AND}{{\rm and }}
\newcommand{\bibauthor}[1]{\textsc{#1},}
\newcommand{\bibjtitle}[1]{\textrm{#1},}
\newcommand{\bibbtitle}[1]{\textit{#1},}
\newcommand{\bibjournal}[1]{\textit{#1}}
\newcommand{\bibvol}[1]{\textbf{#1}}
\newcommand{\bibdoi}[1]{%
	\relax
}
\newcommand{\bibarxiv}[1]{%
	\href{http://arxiv.org/abs/#1}{\texttt{arXiv:#1}.}
}

%
\makeatletter
\def\@biblabel#1{#1.}
\makeatother

%
\makeatletter
\let\Thebibliography=\thebibliography
\renewcommand{\thebibliography}[1]{%
	\def\@mkboth##1##2{}\Thebibliography{#1}%
	\addcontentsline{toc}{section}{References}%
	\setlength{\@topsep}{0pt}
	\setlength{\itemsep}{0pt}%
	\setlength{\parskip}{0pt plus 2pt}%
}
\makeatother

\end{document}